\documentclass[12pt,oneside,reqno]{amsart}
\usepackage{mathrsfs}
\usepackage{amsfonts}
\usepackage{amssymb}
\usepackage{graphicx}
\usepackage{multirow}
\usepackage{float}

\usepackage{ulem}
\usepackage{color}

\numberwithin{equation}{section}\theoremstyle{remark}
\newcommand{\be}{\begin{eqnarray}}
\newcommand{\ee}{\end{eqnarray}}
\newcommand{\ce}{\begin{eqnarray*}}
\newcommand{\de}{\end{eqnarray*}}
\newtheorem{theorem}{Theorem}[section]
\newtheorem{lemma}[theorem]{Lemma}
\newtheorem{remark}[theorem]{Remark}
\newtheorem{definition}[theorem]{Definition}
\newtheorem{proposition}[theorem]{Proposition}
\newtheorem{Examples}[theorem]{Examples}
\newtheorem{corollary}[theorem]{Corollary}

\def\b{\beta}
\def\d{\delta}

\def\l{\lambda}

\def\[{{\Big[}}
\def\]{{\Big]}}
\def\<{{\langle}}
\def\>{{\rangle}}
\def\({{\Big(}}
\def\){{\Big)}}
\def\*{{\times}}

\def\bt{\begin{theorem}}
\def\et{\end{theorem}}
\def\bl{\begin{lemma}}
\def\el{\end{lemma}}
\def\br{\begin{remark}}
\def\er{\end{remark}}
\def\bx{\begin{Examples}}
\def\ex{\end{Examples}}
\def\bd{\begin{definition}}
\def\ed{\end{definition}}
\def\bp{\begin{proposition}}
\def\ep{\end{proposition}}
\def\bc{\begin{corollary}}
\def\ec{\end{corollary}}

\def\mB{{\mathbb B}}

\def\mE{{\mathbb E}}

\def\mN{{\mathbb N}}

\def\mP{{\mathbb P}}

\def\mR{{\mathbb R}}

\def\mW{{\mathbb W}}
\def\mX{{\mathbb X}}

\def\geq{\geqslant}
\def\leq{\leqslant}

\allowdisplaybreaks

\begin{document}

\title{A strong Markov process time-changed  by an inverse killed subordinator}

\dedicatory{$^{\dag}$School of Applied Mathematics, Beijing Normal University Zhuhai,\\
 Guangdong 519087, P.R. China\\
Email: huiyan.zhao@msn.com\\
$^{\dag\dag}$College of Science, North China University of Technology,\\
Beijing 100144, P.R. China\\
Email: xusiyan@ncut.edu.cn}
\author{$^{\dag}$Huiyan Zhao}
\author{$^{\dag\dag}$Siyan Xu}


\date{}

\begin{abstract}
In this paper, we consider a type of time-changed Markov process, where the time-change is an inverse killed subordinator. This can be seen as an extension of Chen (\textit{Chen, Z., Time fractional equations and probabilistic representation, Chaos Solitons and Fractals, 168-174, 2017}). As a result,  it
constructs a one-to-one correspondence between general Bernstein functions (with infinite L\'{e}vy measure) and a class of generalized time-fractional partial differential equations.
\end{abstract}
\maketitle

\section{Introduction}

Let $X=(X_t)_{t\geq 0}$  be a strong Markov process on a separable locally compact Hausdorff space
$\mX$ whose transition semigroup $(T_t)_{t\geq0}$ is a uniformly bounded  and strongly continuous semigroup in some appropriate Banach space $(\mB, \|\cdot\|)$.  For example, $\mB$ can be chosen as $L^p(\mX;m)$ for some measure $m$ on $\mX$ and $p \geq 1$, we refer the reader to \cite{Applebaum,Chen} for more concrete examples. We shall denote by $\mathcal{L}$ the generator of semigroup $(T_t)_{t\geq0}$ and   by $D(\mathcal{L})$ the domain of $\mathcal{L}$.
Let $D=(D_t)_{t\geq 0}$ be a subordinator (i.e. a non-decreasing real-valued L\'{e}vy process) independent of $X$ with $D_0 =0$ and the Laplace exponent $\phi_0$:
                \begin{equation}\label{10_c_1}
                \phi_0(\l) =  k\l + \int_0^\infty \big{(}1- e^{-\l z} \big{)}\,\mu(dz),
                \end{equation}
such that
                \begin{equation}\label{10_c_2}
                  \mE (e^{-\l D_t}) = e^{-t \phi_0(\l) }, \quad \l>0,
                \end{equation}
where $k\geq 0$ and $\mu$ is a L\'{e}vy measure satisfying $\mu(-\infty, 0) =0$ and $\int_0^\infty (1\wedge z)\,\mu(dz) < \infty$.
Let $E=(E_t)_{t\geq 0}$ be the general inverse of $D$ defined as
                \begin{equation*}
                 E_t:= \inf\{s>0: D(s) >t\}, \quad t\geq0.
                \end{equation*}
 We shall call $E$ a time-change and the composite process $X_E=(X_{E_t})_{t\geq 0}$
a time-changed Markov process. 

Recent years, time-changed Markov processes have attracted many researchers due to their deep connections with the time-fractional Kolmogorov equations (or time-fractional Fokker-Planck equations), where the latter often appeared as an important tool to model complex anomalous diffusions, see, e.g. \cite{Hahn,Hahn3,Kobayashi,Meerschaert,Chen} and references therein.
For $f \in \mB$,  let
                \begin{equation}\label{10_solution_1}
                  v(t,x) := \mE [T_{E_t} f(x)] =\mE (f(X_{E_t})|X_0=x).
                \end{equation}
The following are some known relationships between time-changed Markov processes and  time-fractional partial differential equations.
\begin{itemize}
  \item[(i)]  Firstly, assume that $D$  is  a $\b$-stable subordinator ($0 < \b <1$). That is, $D$ is a special subordinator with Laplace exponent
                    \begin{equation*}
                      \phi_0(\l) = \l^\b.
                    \end{equation*}
Then, $v$ (defined as in (\ref{10_solution_1})) is the unique solution of the following time-fractional Cauchy problem
                \begin{equation*}
                  \partial_t^{\b} v = \mathcal{L} v,  \quad  v(0,x)=f(x),
                \end{equation*}
where $\partial_t^{\b}$ is the Caputo type fractional derivative of order $\b$ defined as
                \begin{equation*}
                  \partial_t^{\b} g(t) := \frac{1}{\Gamma(1-\b)} \frac{d}{dt} \int_0^t
                  (t-s)^{-\b}(g(s) - g(0))\,ds.
                \end{equation*}
In this situation, we refer the reader to \cite{Baeumer} for the case where the semigroup $(T_t)_{t\geq0}$ is generated by a L\'{e}vy process.

 \item[(ii)] Next, let $\nu$ be a finite measure on $(0,1)$ with $\nu(0,1)>0$. Assume that
 $D$ is a subordinator  with Laplace exponent
                    \begin{equation*}
                      \phi_0(\l) = \int_0^1\l^{\b} \,\nu(d\b).
                    \end{equation*}
 We note that such $D$ can be constructed by a weighted mixture of independent stable subordinators, see, e.g., \cite{Hahn}.
 Then, $v$ (defined as in (\ref{10_solution_1})) is the unique solution to the following abstract time-fractional Cauchy problem
                \begin{equation*}
                  \partial_t^{v} v = \mathcal{L} v,  \quad  v(0,x)=f(x),
                \end{equation*}
where $\partial_t^{\nu}$ is the distributed-order  derivative defined by
                \begin{equation*}
                  \partial_t^{\nu} g(t) := \int_0^1 \partial_t^{\b}g(t)\, \nu(d\b).
                \end{equation*}
An important application of distributed-order derivative is to model ultrslow diffusion, we refer the reader to \cite{Meerschaert1,Meerschaert2,Hahn} for the related topics.
 \item[(iii)] Recently, Chen in \cite{Chen} considered  a more general time-fractional derivative. To be clear, let $\mW$ be the set of
 functions $w:(0,\infty) \to [0,\infty)$, which are right continuous, unbounded,  non-increasing and locally integrable on $[0,\infty)$ such that
            \begin{equation*}
            \lim_{z \to \infty} w(z) =0
            \end{equation*}
 and
            \begin{equation*}
            \int_0^\infty (1\wedge z) (-dw(z)) < \infty.
            \end{equation*}
For a $w \in \mW$,  a generalized time-fractional derivative is defined for suitable $g$ as
                \begin{equation*}
                  \partial_t^{w} g(t) := \frac{d}{dt} \int_0^t w(t-s) (g(s) - g(0))\,ds.
                \end{equation*}
 It is shown  in \cite{Chen} that
 there exists a one-to-one correspondence between such generalized time-fractional derivatives and general driftless subordinators with infinite L\'{e}vy measure (i.e.,  a subordinator defined as in (\ref{10_c_1})-(\ref{10_c_2}) with $k=0$ and $\mu(0,\infty) = \infty$). Concretely speaking,
  for every $w \in \mW$, let $\mu$ be a measure on $(0, \infty)$ such that
                \begin{equation*}
                w(x) = \mu((x,\infty)).
                \end{equation*}
  It is clear that
                \begin{equation*}
                  \mu((0,\infty)) = \infty, \quad \text{and} \quad
                  \int_0^\infty (1\wedge z ) \mu(dz) < \infty.
                \end{equation*}
  Then, there exists a unique driftless subordinator (in the distributional sense) with  such infinite L\'{e}vy measure $\mu$. Conversely, given a driftless subordinator with  infinite L\'{e}vy measure $\mu$.
  One can define
                \begin{equation*}
                  w(x):= \mu((x,\infty))
                \end{equation*}
  such that $w \in \mW$.

  Next, assume that
  $D$ is general subordinator satisfying (\ref{10_c_1}) and (\ref{10_c_2}) with $k\geq 0$ and $\mu((0,\infty)) = \infty$. Under the framework of this generalized  time-fractional derivative, Chen in \cite{Chen} proved an interesting result, i.e., $v$ (defined as in (\ref{10_solution_1})) is the unique solution in $(\mB,\|\cdot\|)$ to the following time-fractional equation
                 \begin{equation}\label{10_equation}
                k \partial_t v  +  \partial_t^{w} v = \mathcal{L} v,  \quad  v(0,x)=f(x),
                \end{equation}
  in an appropriate sense.
\end{itemize}

We recall that, for a  given subordinator $D$, its Laplace exponent
  $\phi_0(\l)$ defined as in (\ref{10_c_1}) is a Bernstein function (see Section 2 for the definition) satisfying $\lim_{\l \to 0} \phi_0(\l) = 0$.  In other words,
  the results of Chen \cite{Chen} have constructed a one-to-one correspondence from the Bernstein function $\phi_0$
  with $\lim_{\l \to 0} \phi_0(\l) = 0$ to the time-fractional equation (\ref{10_equation}).
Now, given a more general Bernstein function $\phi(\l)$ with
                \begin{equation*}
                  \lim_{\l \to 0} \phi(\l) \neq 0.
                \end{equation*}
Does there exist a similar correspondence?

It's lucky that, for a given Bernstein function $\phi(\l)$,  there exists a unique killed subordiator $D^S$ (see, e.g., p.56 in \cite{Applebaum} or Section 2) whose Laplace exponent is exact $\phi(\l)$. Then, 
 let $E^S$ be the general inverse of $D^S$ and set
                \begin{equation*}
                 u(t,x):=\mE \big{(} T_{E_t^S} f(x) \big{)},
                \end{equation*}
 for $x\in \mX$, $f \in \mB$ and $t\geq 0$.
 Inspired by the work of Chen \cite{Chen}, we need to consider the question:
 what equation does $u(t,x)$ satisfy?
The main result is presented in Theorem 2.1 in Section 2, which can be seen as an extension of Chen \cite{Chen} to the killed subordinator case. As a result,  this
constructs a one-to-one correspondence between general Bernstein functions (with infinite L\'{e}vy measure) and a class of generalized time-fractional partial differential equations.
We shall follow the idea of Chen, with some crucial changes in the proof. As in
Chen \cite{Chen}, the proofs also work for uniformly bounded and continuous semigroups defined on some Banach spaces.

 The rest of this paper is organized as follows.
In Section 2, we present some basic concepts and our main result. The proofs are given in Section 3. Finally, an example is presented in Section 4.

\section{Preliminaries and the main result}

A function $\phi:(0,\infty) \to [0,\infty)$ is called a Bernstein function if
it is smooth and satisfies $(-1)^n \phi^{(n)}(\l) \leq 0$ for every $\l>0$ and $n \in \mN$.
It's known that every Bernstein function  $\phi$ admits a unique representation (see, e.g., Theorem 1.3.23 in \cite{Applebaum})
               \begin{equation*}
                \phi(\l) = a + k\l + \int_0^\infty \big{(}1- e^{-\l z} \big{)}\,\mu(dz),
                \end{equation*}
 where $a,\, k\geq 0$ and $\mu$ is a L\'{e}vy measure (i.e., a nonnegative Radon measure on $(0, \infty)$) with $\int_0^\infty (1\wedge z)\,\mu(dz) < \infty$. In the following, we shall call the triple $(a,k,\mu)$ as the characteristics of $\phi$.
Next, given a Bernstein function $\phi$ with the characteristics $(a,k,\mu)$, it's known that there exists a  killed subordinator $D^S$ defined as
                \begin{equation*}
                 D^S(t) := \left\{
                     \begin{array}{ll}
                       D_t, &  t < S, \\
                       \infty, &  t \geq S,
                     \end{array}
                   \right.
                \end{equation*}
such that
                \begin{equation*}
                  \mE(e^{-\l  D^S_t }) = e^{-t  \phi(\l)},
                \end{equation*}
where $D=(D_t)_{t\geq0}$ is a subordinator satisfying (\ref{10_c_1}) and (\ref{10_c_2}),
 $S$ is an exponentially distributed random variable independent of $D$ with the density function
 $g(z) = a e^{-az}$ for $z> 0$ (see, e.g., p.56 in \cite{Applebaum}).

Let $E^S$ be the general inverse of the killed subordinator $D^S$, that is,
                \begin{equation*}
                 E_t^S:= \inf\{s>0: D^S(s) >t\}, \quad t\geq0.
                \end{equation*}
The main result of this paper is the following.

\begin{theorem}\label{10_th_1}
Let $\phi$ be a Bernstein function with the characteristics $(a,k,\mu)$
and $D^S$ be the corresponding killed subordinator with its inverse $E^S$ defined as above.
Under this  setting, suppose that $a>0$ and $\mu((0,\infty)) = \infty$ and let $w(z):=\mu((z,\infty))$ for $z>0$. Then, for every $f \in D(\mathcal{L})$,  the function
                \begin{equation*}
                  u(t,x): = \mE \big{(} T_{E_t^S} f(x)\big{)}, \quad t\geq 0,
                \end{equation*}
is the unique solution  in $(\mB,\|\cdot\|)$ of the following time-fractional equation
                \begin{equation}\label{10_equation_1}
                 \left\{
                     \begin{array}{ll}
                       (k \partial_t + \partial_t^w) u(t,x) = (\mathcal{L} -a) u(t,x) +  a f(x), &  t>0, \\
                       u(0,x) = f(x),
                     \end{array}
                   \right.
                \end{equation}
in the  sense:
            \begin{enumerate}
              \item[(i)]  $\sup_{t\geq0} \|u(t, \cdot)\| < \infty $;
              \item[(ii)]  $x \to u(t,x)$ is in $D(\mathcal{L})$ for  each $t\geq 0$ with
              $\sup_{t\geq0} \|\mathcal{L}u(t, \cdot)\| < \infty $;
              \item[(iii)] both $t \to u(t,\cdot)$ and $t \to \mathcal{L}u(t, \cdot)$ are continuous in $(\mB, \|\cdot\|)$;
              \item[(iv)]  for every $t>0$,
                \begin{eqnarray*}
                   && \lim_{\d \to 0 } \frac{1}{\d} \Big{(}  k \big{(}u(t+\d,x) - u(t,x)\big{)} + \int_t^{t+\d} w(t-s) \big{(} u(s,x) - u(0,x) \big{)}\,ds\Big{)}\\
                   &=& \mathcal{L} u(t,x) - a u(t,x) + f(x) \quad \text{in $(\mB, \|\cdot\|)$.}
                \end{eqnarray*}
            \end{enumerate}
\end{theorem}

\begin{remark}
The followings are two remarks for our main result.
\begin{enumerate}
  \item [(i)] A killed subordinator can be seen as an extension of a general subordinator.
  Indeed, if the exponential parameter $a$ of $S$ degenerates to 0, then the related killed subordinator $D^S$ will become an ordinary subordinator. Hence, in this sense, the result of Chen \cite{Chen} (see, e.g., equation \ref{10_equation}) can be covered by Theorem \ref{10_th_1} as a special case.
  \item [(ii)] As we can see, there exists a one-to-one correspondence between the Bernstein function $\phi$  and equation (\ref{10_equation_1}). This reveals a kind of relationship between general Bernstein functions (with infinite L\'{e}vy measure) and generalized time-fractional partial differential equations.
\end{enumerate}

\end{remark}

\section{Proof of the main result}
In this section, we will prove the main result by following the method of Chen \cite{Chen}. Firstly, it's well known that, for $f \in D(\mathcal{L})$,
            \begin{equation}\label{10_a_2}
              \frac{d}{dt} T_t f(x) = \mathcal{L}T_t f = T_t \mathcal{L} f
            \end{equation}
in $(\mB, \|\cdot\|)$.
Throughout this section, $w$ and $u$ will denote the functions defined as in  Theorem \ref{10_th_1}.
Next, recall that $D$ is a subordinator satisfying (\ref{10_c_1}) and (\ref{10_c_2}).  By using the L\'{e}vy-It\^{o} decomposition, we have
            \begin{equation}\label{10_a_1}
              D_t = kt + \bar{D}_t, \quad t \geq 0,
            \end{equation}
where $\bar{D}=(\bar{D}_t)_{t\geq 0}$ is a driftless subordinator with
            \begin{equation*}
              \mE(e^{-\l \bar{D}_t }) = e^{- t \int_0^\infty \big{(}1- e^{-\l z} \big{)}\,\mu(dz)}, \quad \l>0.
            \end{equation*}
From (\ref{10_a_1}), we have
            \begin{equation}\label{10_10}
            \bar{D}_t = D_t - kt.
            \end{equation}

The following result is taken from Chen \cite{Chen}.

\begin{lemma}\label{10_lemma_1}
There exists a Borel null set $ A \subset (0,\infty)$ such that for every
$r \in (0,\infty)$ and $ t \in (0,\infty)\setminus A$,
        \begin{equation*}
          \mP \big{(} \bar{D}_r  \geq t \big{)} =
          \int_0^r \mE \big{[}w(t-\bar{D}_y) 1_{\{t > \bar{D}_y\}}\big{]}\,dy
        \end{equation*}
and
        \begin{equation*}
          \mP \big{(} \bar{D}_r = t \big{)} = 0.
        \end{equation*}
\end{lemma}

\begin{remark}\label{10_re_1} \

\begin{itemize}
  \item[(i)] The assumption $\mu(0,\infty) = \infty$ is to make sure $\bar{D}$ is strictly increasing almost surely. Hence $D$ is also strictly increasing almost surely, as a consequence, $E$ and $E^S$ are continuous almost surely. This assumption is indispensable for Lemma \ref{10_lemma_1}, we refer the reader to Lemma 2.1 in \cite{Chen} for more details.
  \item[(ii)] According to Lemma \ref{10_lemma_1} and equation (\ref{10_10}), there exists a Borel null set $B \subset (0, \infty)$ such that, for every
$r \in (0,\infty)$ and $ t \in (0,\infty)\setminus B$, we  have
         \begin{equation*}
          \mP \big{(} D_r = t \big{)} = 0.
        \end{equation*}
\end{itemize}

\end{remark}
Next, recall that
        \begin{equation*}
          E_t =  \inf\{s>0:  D_s > t\}.
        \end{equation*}
By the definition of $E^S$, it's easy to see that
        \begin{equation}\label{10_e_1}
          E_t^S = E_t  \wedge S.
        \end{equation}
We need the following representation for the distribution of $E^S_s$ for $s\geq 0$.

\begin{lemma}
For every $s,\, r \geq 0$, we have
        \begin{equation}\label{10_1}
          \mP(E_s^S \leq r) =
          1 - e^{-ar} \big{(}1-\mP(D_r  \geq s ) \big{)}.
        \end{equation}
Moreover,  we also have
        \begin{equation}\label{10_9}
         \mP(E_s^S \leq r) = 1 - e^{-ar} \big{(}1-\mP(\bar{D}_r  \geq s-kr) \big{)}.
        \end{equation}
\end{lemma}
\begin{proof}
Note that $D$ and $S$ are independent. Then,
according to (\ref{10_e_1}), we have
            \begin{eqnarray*}
          \mP(E_s^S \leq r)
          &=&  1 - \mP(E_s \wedge  S > r) \nonumber  \\
          &=&  1 - \mP(E_s  > r) \mP( S > r) \nonumber  \\
          &=&  1 - e^{-ar} \big{(}1-\mP(E_s  \leq r) \big{)} \nonumber  \\
          &=&  1 - e^{-ar} \big{(}1-\mP(D_r  \geq s ) \big{)},
        \end{eqnarray*}
which implies (\ref{10_1}). Next, (\ref{10_9}) follows by using (\ref{10_10}). We finish the proof.
\end{proof}

Similarly as in \cite{Chen}, let us define
$G(0)=0$ and
        \begin{equation*}
          G(x) = \int_0^x w(t)\,dt
        \end{equation*}
for $x>0$. Then, $G(x)$ is a continuous function on $[0,\infty)$ with $G'(x) = w(x)$ on $(0,\infty)$. It follows that, by using (\ref{10_1}), Remark \ref{10_re_1} (ii) and  the integration by parts formula, for every $t, \, r\geq 0$,
        \begin{eqnarray}\label{10_2}
           && \int_0^t w(t-s) \mP(E_s^S \leq r)\,ds \nonumber \\ \nonumber
           &=& -\int_0^t  \mP(E_s^S \leq r)\,dG(t-s) \\ \nonumber
           &=& G(t) + \int_0^t  G(t-s) \,d_s \mP(E_s^S \leq r) \\ \nonumber
           &=& G(t) - e^{-ar}\int_0^t  G(t-s) \,d_s \mP(D_r  \leq  s )  \\
             &=& G(t) - e^{-ar} \mE [G(t-D_r) 1_{\{t\geq D_r\}}].
        \end{eqnarray}

We also need  the following result.
\begin{lemma} For the $G$ defined above, we have
        \begin{eqnarray}\label{10_0}
           &&  \int_0^t w(t-s) \big{(} u(s,x) - u(0,x) \big{)}\,ds \nonumber\\
           &=&  \int_0^\infty
           e^{-ar} \mE [G(t-D_r) 1_{\{t > D_r\}}] \mathcal{L} T_rf(x) \,dr.
        \end{eqnarray}
\end{lemma}

\begin{proof}
By using the Fubini's theorem, we have
        \begin{eqnarray*}
           &&  \int_0^t w(t-s) \big{(} u(s,x) - u(0,x) \big{)}\,ds \\
           &=& \int_0^t w(t-s) \Big{(} \int_0^\infty \big{(}T_rf(x) - f(x)\big{)}\,d_r\mP(E_s^S \leq r)\Big{)}\,ds \\
           &=&  \int_0^\infty \big{(}T_rf(x) - f(x)\big{)}\,d_r\Big{(}
           \int_0^t w(t-s) \mP(E_s^S \leq r)\,ds\Big{)}.
        \end{eqnarray*}
By using (\ref{10_2}) and the integration by parts formula again, we get
        \begin{eqnarray*}
           &&  \int_0^t w(t-s) \big{(} u(s,x) - u(0,x) \big{)}\,ds \nonumber\\
           &=&  -\int_0^\infty \big{(}T_rf(x) - f(x)\big{)}\,d_r\Big{(}
           e^{-ar} \mE [G(t-D_r) 1_{\{t\geq D_r\}}]\Big{)} \nonumber \\
           &=&  \int_0^\infty
           e^{-ar} \mE [G(t-D_r) 1_{\{t\geq D_r\}}] \mathcal{L} T_rf(x) \,dr \nonumber \\
           &=&  \int_0^\infty
           e^{-ar} \mE [G(t-D_r) 1_{\{t > D_r\}}] \mathcal{L} T_rf(x) \,dr.
        \end{eqnarray*}
We finish the proof.
\end{proof}

\vspace{1mm}

Now, we are in a position to give the proof of our main result.

\vspace{1mm}
\noindent\textbf{\textit{Proof of Theorem \ref{10_th_1}}.} We will prove the existence firstly, and then prove the uniqueness.

(i)  (Boundedness an continuity) Firstly, recall that
                \begin{equation*}
                  u(t,x)= \mE \big{(} T_{E_t^S} f(x)\big{)}
                  =\int_0^\infty T_{r} f(x)\, d_r \big{(}\mP (E_t^S \leq r)\big{)}.
                \end{equation*}
Then, for $f \in D(\mathcal{L})$, we have
                \begin{eqnarray}\label{10_e_2}
                   &&  \sup_{t\geq0} \|u(t, \cdot)\|
                   = \sup_{t\geq0} \|\int_0^\infty T_{r} f(\cdot)\, d_r \mP\big{(} E_t^S \leq r\big{)}\| \nonumber \\
                   &\leq & \sup_{r\geq0} \|T_{r} f(\cdot) \|
                    \leq M \|f\| < \infty,
                \end{eqnarray}
where $M$ is a bound of $(T_t)_{t\geq 0}$.

In the following, we aim to prove, for fixed $t\geq 0$,
            \begin{equation}\label{10_s}
              u(t,\cdot) \in D(\mathcal{L}) \quad \text{and} \quad
               \mathcal{L} u(t,x) = \mE \big{(} T_{E_t^S} \mathcal{L}f(x)\big{)}.
            \end{equation}
For this, let $t\geq 0$ be fixed,  $F_t(r)$ be the distribution function of $E^S_t$  and  $(F_t^n(r))_{n\geq 1}$ be a sequence of simple increasing  and right continuous functions such that
                \begin{equation*}
                  \lim_{n \to \infty } \sup_{r \geq 0}|F_t^n(r) - F_t(r)|=0
                \end{equation*}
and
                \begin{equation*}
                    F_{n,t}(r)  \leq F_t(r), \quad \text{for all $r\geq 0$.}
                \end{equation*}
Let
            \begin{equation*}
                  u_n(t,x)=  \int_0^\infty T_{r} f(x)\, dF_t^n(r).
            \end{equation*}
It follows that,
            \begin{eqnarray*}
               &&  \|u_n(t,\cdot) - u(t,\cdot)\| \\
               &\leq & \sup_{r\geq 0}\|T_r f\| \sup_{r \geq 0}|F_t^n(r) - F_t(r)|\\
               &\to& 0,
            \end{eqnarray*}
as $n \to \infty$.
Now, due to (\ref{10_a_2}), we immediately get $u_n(t, \cdot) \in D(\mathcal{L})$ and
                \begin{equation*}
                  \frac{d}{dt} u_n(t,x) = \mathcal{L} u_n(t,x) = \int_0^\infty T_{r} \mathcal{L}f(x)\, dF_{n,t}(r).
                \end{equation*}
Moreover, since
                \begin{equation*}
                  \lim_{n \to \infty} \mathcal{L} u_n(t,x) =  \int_0^\infty T_{r} \mathcal{L}f(x)\, dF_{t}(r) = \mathcal{L} u(t,x).
                \end{equation*}
Then, by using the closed property of operator $\mathcal{L}$, we deduce that $u(t,\cdot) \in D(\mathcal{L})$ and
                \begin{equation}\label{10_add_1}
                  \mathcal{L} u(t,x) = \int_0^\infty T_{r} \mathcal{L}f(x)\, dF_{t}(r)
                  =\mE \big{(} T_{E_t^S} \mathcal{L}f(x)\big{)}.
                \end{equation}
This completes the proof of statement (\ref{10_s}).
Next, by taking a similar proof as in (\ref{10_e_2}), we also have
                \begin{eqnarray*}
                   &&  \sup_{t\geq0} \|\mathcal{L} u(t, \cdot)\| < \infty.
                \end{eqnarray*}

Finally, since $(T_t)_{t\geq 0}$ is a strongly continuous semigroup on $\mB$ with $\sup_{t\geq 0} \|T_t\| < \infty$ and $t \to E^S_t$ is continuous almost surely (see Remark \ref{10_re_1} (i)). Then, by applying the dominated convergence theorem $t \to u(t, \cdot) $ and
$t \to \mathcal{L} u(t,\cdot)$ are continuous in $(\mB, \|\cdot\|)$.

\vspace{2mm}

(ii) (Existence)
In the following, we prove that $u$  is a solution of our equation. Firstly,
by Lemma \ref{10_lemma_1}, (\ref{10_1}) and (\ref{10_9}), we have $\mP(E_s^S \leq r) =1$ for $ s \leq kr$ and
        \begin{equation*}
                  \mP(E_s^S \leq r)=
                       1 - e^{-ar}
         \Big{(}1-\int_0^r \mE \big{[}w(s-kr-\bar{D}_y)
         1_{\{s-kr > \bar{D}_y\}}\big{]}\,dy\Big{)}
        \end{equation*}
for $s> kr$.  It follows that,
        \begin{eqnarray}\label{10_3}
          && \int_0^t  \mP(E_s^S \leq r)\,ds \nonumber \\
          &=&  t\wedge(kr) + \int_{t\wedge(kr)}^t  \mP(E_s^S \leq r)\,ds \nonumber \\
           &=&  t\wedge(kr) + 1_{\{kr <t \}}   (t -kr)(1-e^{-ar}) \nonumber \\
           && + \,1_{\{kr <t \}} e^{-ar}
            \mE \int_0^r\Big{(}\int_{kr}^t w(s-kr-\bar{D}_y) 1_{\{s-kr > \bar{D}_y\}}\,ds\Big{)}\,dy  \nonumber\\
             &=&  t\wedge(kr) + 1_{\{kr <t \}}   (t -kr)(1-e^{-ar}) \nonumber\\
           && + \,1_{\{kr <t \}} e^{-ar}
            \mE \int_0^r G(t-kr-\bar{D}_y) 1_{\{t-kr > \bar{D}_y\}}\,dy.
        \end{eqnarray}
Then, by (\ref{10_add_1}), (\ref{10_3}) and Fubini's theorem, we get
        \begin{eqnarray}\label{10_5}
           && \int_0^t  \mathcal{L} u(s,x) \,ds \nonumber \\
           &=& \int_0^t \Big{(} \int_0^\infty T_r\mathcal{L}f(x)\,d_r\mP(E_s^S \leq r)  \Big{)} \,ds \nonumber \\
           &=& \int_0^\infty T_r\mathcal{L}f(x)\,d_r\Big{(} \int_0^t  \mP(E_s^S \leq r)\,ds  \Big{)} \nonumber \\
           &=& \mE \int_0^{t/k} T_r\mathcal{L}f(x)\,\Big{(} k - k (1-e^{-ar}) + a e^{-ar} (t -kr) \nonumber \\
           &&-\, a e^{-ar}
            \int_0^r G(t-kr-\bar{D}_y) 1_{\{t-kr > \bar{D}_y\}}\,dy \nonumber \\
            && +\, e^{-ar}
             G(t-kr-\bar{D}_r) 1_{\{t-kr > \bar{D}_r\}}  \nonumber \\
            && -\, k e^{-ar}
            \int_0^r w(t-kr-\bar{D}_y) 1_{\{t-kr > \bar{D}_y\}}\,dy  \Big{)} \,dr \nonumber\\
            &=:& I_1 + I_2 +I_3,
        \end{eqnarray}
where
        \begin{eqnarray*}
          I_1 &:=& \mE \int_0^{t/k} T_r\mathcal{L}f(x) e^{-ar}
             G(t-kr-\bar{D}_r) 1_{\{t-kr > \bar{D}_r\}} \,dr,
        \end{eqnarray*}
        \begin{eqnarray*}
          I_2 &:=& k  \mE\int_0^{t/k} T_r\mathcal{L}f(x) e^{-ar}\big{(} 1- \int_0^r w(t-kr-\bar{D}_y) 1_{\{t-kr > \bar{D}_y\}}\,dy \big{)} \,dr,
        \end{eqnarray*}
and
        \begin{eqnarray*}
          I_3 &:=&  \mE \int_0^{t/k} T_r\mathcal{L}f(x)\,\Big{(}  a e^{-ar} (t -kr) \\
           &&-\, a e^{-ar}
            \int_0^r G(t-kr-\bar{D}_y) 1_{\{t-kr > \bar{D}_y\}}\,dy \Big{)} dr.
        \end{eqnarray*}
For $I_1$, we have
        \begin{eqnarray*}
          I_1 &=&  \mE \int_0^{t/k} T_r\mathcal{L}f(x) e^{-ar}
             G(t-kr-\bar{D}_r) 1_{\{t-kr > \bar{D}_r\}} \,dr \nonumber\\
             &=&  \int_0^{t/k} T_r\mathcal{L}f(x) e^{-ar}\mE[
             G(t-D_r) 1_{\{t> D_r\}}] \,dr \nonumber\\
             &=&  \int_0^{\infty} T_r\mathcal{L}f(x) e^{-ar}\mE[
             G(t-D_r) 1_{\{t> D_r\}}] \,dr.
        \end{eqnarray*}
It follows that, by using (\ref{10_0}), we get
        \begin{equation}\label{10_6}
          I_1 = \int_0^t w(t-s) \big{(} u(s,x) - u(0,x) \big{)}\,ds.
        \end{equation}
For $I_2$, by using the Fubini's theorem and the integration by parts formula, we get
        \begin{eqnarray}\label{10_7}
          I_2 &=& k  \int_0^{t/k} T_r\mathcal{L}f(x) e^{-ar}\big{(} 1- \mE\int_0^r w(t-kr-\bar{D}_y) 1_{\{t-kr > \bar{D}_y\}}\,dy \big{)} \,dr \nonumber  \\
          &=& k  \int_0^{t/k} T_r\mathcal{L}f(x) \big{(} 1- \mP(E_t \leq r) \big{)} \,dr \nonumber \\
          &=& k  \int_0^{\infty}  \big{(} 1- \mP(E_t \leq r) \big{)} \,dr(T_rf(x) -f(x))
          \nonumber \\
          &=& k  \int_0^{\infty} (T_rf(x) -f(x))\,  d_r \mP(E_t \leq r) \nonumber \\
          &=& k (u(t,x) - u(0,x))
        \end{eqnarray}
Next, since
        \begin{eqnarray*}
           &&  \int_0^t u(s,x)\,ds \\
           &=& \int_0^t\Big{(}\int_0^\infty T_rf(x)\,d_r\mP(E_s \leq r)\Big{)}\,ds \\
           &=& \int_0^\infty T_rf(x)\,d_r\Big{(}\int_0^t\mP(E_s \leq r)\,ds\Big{)} \\
           &=& \int_0^{\infty} T_rf(x)\,d_r\Big{(}\int_0^t\mP(E_s \leq r)\,ds - t\Big{)} \\
           &=&  f(x) t - \int_0^\infty \big{(}\int_0^t\mP(E_s \leq r)\,ds - t\big{)}\,d_r( T_rf(x)) \\
           &=&  f(x) t - \int_0^\infty T_r\mathcal{L}f(x)\big{(}\int_0^t\mP(E_s \leq r)\,ds - t\big{)}\,dr,
        \end{eqnarray*}
where we have used the integration by parts formula in the third equality. Furthermore, by (\ref{10_3}), we obtain
        \begin{eqnarray*}
           &&  \int_0^t u(s,x)\,ds  \nonumber \\
           &=&  f(x) t - \int_0^\infty T_r\mathcal{L}f(x)\big{(}\int_0^t\mP(E_s \leq r)\,ds - t\big{)}\,d_r \nonumber \\
           &=&  f(x) t - \int_0^{t/k}T_r\mathcal{L}f(x)\big{(}\int_0^t\mP(E_s \leq r)\,ds - t\big{)}\,d_r \nonumber \\
           &=&  f(x) t + \mE \int_0^{t/k}T_r\mathcal{L}f(x)e^{-ar}\Big{(} (t -kr) \nonumber\\
           && - \,
             \int_0^r G(t-kr-\bar{D}_y) 1_{\{t-kr > \bar{D}_y\}}\,dy\Big{)}\,d_r \nonumber \\
           &=&  f(x) t + I_3/a.
        \end{eqnarray*}
Therefore, we get
        \begin{equation}\label{10_8}
          I_3 =  a \big{(}\int_0^t u(s,x)\,ds - f(x) t \big{)}.
        \end{equation}
Now, combining (\ref{10_5})-(\ref{10_8}) together, for every $t>0$, we obtain
        \begin{eqnarray*}
           \int_0^t  \mathcal{L} u(s,x) \,ds &=&
           \int_0^t w(t-s) \big{(} u(s,x) - u(0,x) \big{)}\,ds \\
           && + \, k (u(t,x) - u(0,x))  + a \big{(}\int_0^t u(s,x)\,ds - f(x) t \big{)},
        \end{eqnarray*}
which implies that
        \begin{equation*}
            (k \partial_t + \partial_t^w) u(t,x) = (\mathcal{L} -a) u(t,x) + a f(x)
        \end{equation*}
by using the continuity of $t \to u(t, \cdot)$ and  $t \to \mathcal{L}u(t,\cdot)$  in $(\mB, \|\cdot\|)$.

\vspace{2mm}
(iii) (Uniqueness) Suppose that $\tilde{u}(t,x)$ is another solution of equation (\ref{10_equation_1})  with  $\tilde{u}(0,x) =f(x)$ in the sense of Theorem \ref{10_th_1}. It follows that
$v(t,x):= \tilde{u}(t,x) - u(t,x)$ is a solution of the following homogenous equation
            \begin{equation*}
                 \left\{
                     \begin{array}{ll}
                       (k \partial_t + \partial_t^w) v(t,x) = (\mathcal{L} -a) v(t,x), &  t>0, \\
                       v(0,x) = 0.
                     \end{array}
                   \right.
                \end{equation*}
Hence, for every $t>0$, we have
            \begin{equation}\label{10_11}
              k\, v(t,x) +  \int_0^t w(t-s)v(s,x)\,ds = \int_0^t (\mathcal{L} -a) v(s,x)\,ds.
            \end{equation}
Next, let $V(\l,x):= \int_0^\infty e^{-\l t} v(t,x)\,dt$ be the Laplace transform of $t \to v(t,x)$.
It is easy to see that $V(\l, \cdot) \in \mB$ for every $\l>0$ and
            \begin{equation*}
              \|V(\l, \cdot)\| \leq  \frac{1}{\l} \sup_{t\geq 0} \|v(t,\cdot)\|.
            \end{equation*}
In addition, by taking a similar procedure as in the proof of (\ref{10_s}), we have, for every $\l >0$,  $V(\l,\cdot) \in D(\mathcal{L})$,
            \begin{equation*}
              \mathcal{L} V(\l, \cdot) = \int_0^\infty e^{-\l t} \mathcal{L}v(t,\cdot)\,dt
            \end{equation*}
and
            \begin{equation*}
              \|\mathcal{L} V(\l, \cdot)\| \leq \int_0^\infty e^{-\l t} \|\mathcal{L}v(t,\cdot)\|\,dt
              \leq \frac{1}{\l} \sup_{t \geq 0}\|\mathcal{L}v(t,\cdot)\|.
            \end{equation*}
Now, by taking Laplace transform on both sides of (\ref{10_11}), we get
            \begin{equation}
             V(\l,x) \Big{(} k+  \int_0^\infty e^{-\l s} w(s)\,ds \Big{)} = \frac{1}{\l}(\mathcal{L} -a ) V(\l,x).
            \end{equation}
It follows that
            \begin{equation*}
              \mathcal{L}  V(\l,x) = V(\l,x) \Big{(} \big{(}k  +  \int_0^\infty e^{-\l s} w(s)\,ds \big{)}\l + a\Big{)} = \phi(\l) V(\l,x) .
            \end{equation*}
That is, we have obtained
            \begin{equation*}
             \big{(} \phi(\l) - \mathcal{L} \big{)}\, V(\l,x) =0.
            \end{equation*}
Next, recall that, for every $\l>0$, the resolvent $(\l - \mathcal{L})^{-1}$ exists (see, e.g., p.159 in \cite{Applebaum}). Hence, $V(\l, \cdot) = 0$ for every $\l >0$. Finally, according to the uniqueness of the Laplace transform, we obtain $v(t, \cdot) = 0$ in $\mB$ for every $t>0$. Therefore, $\tilde{u}(t,\cdot) = u(t, \cdot)$ in $\mB$ for every $t\geq 0$. We finish the proof.

\section{An example}

In this section, we present an example to explain our result. Let $B=(B_t)_{t\geq 0}$ be a $\mR^d$-valued standard Brownian motion. Let $N(dt,dx)$ be a Poisson random measure on
 $(\mR^{+} \times (\mR^d- \{0\}) )$ with intensity measure $\nu(dx)$ and denoted  its compensator by $\tilde{N}(dt,dx):= N(dt,dx) - \nu(dx)\,dt$.
We consider the $\mR^d$-valued process Markov $Y=(Y_t)_{t\geq 0}$  which is the unique strong solution of the following equation
                \begin{eqnarray*}
                  dY_t &=& b(Y_{s-})\,ds + \sigma(Y_{s-})\,dB_s + \int_{\|x\|< 1} F(Y_{s-},x)\tilde{N}(ds,dx) \\
                    &&    +\, \int_{\|x\|\geq1} G(Y_{s-},x)N(ds,dx),
                \end{eqnarray*}
where $b: \mR^d \to \mR^d$, $\sigma: \mR^d \to \mR^{d\times d}$, $F: \mR^d \times \mR^d \to \mR^d$  and $G: \mR^d \times \mR^d \to \mR^d$ are measurable functions satisfying the classical Lipschitz and linear growth conditions (see, e.g., p. 365 in \cite{Applebaum}). For $f \in C_0(\mR^d)$, let
                \begin{equation*}
                  T_t f(y):= \mE (f(Y_t)|Y_0 = y).
                \end{equation*}
Then $(T_t)_{t\geq}$ is the semigroup in the Banach space $C_0(\mR^d)$. Its
generator is
                \begin{eqnarray*}
                  \mathcal{A} f &=& b^i(y) (\partial_i f)(y) +
                               \frac{1}{2} a^{ij}(y) (\partial_i\partial_j f)(y)\\
                               &&  +\, \int_{|x|<1} \big{(} f( F(y,x) + y) -f(y)
                               - F^{i}(y,x)(\partial_i f)(y) \big{)}\,\nu(dx)\\
                               && + \, \int_{|x|\geq 1} \big{(}
                               f(G(y,x)+y) -f(y) \big{)} \,\nu(dx).
                \end{eqnarray*}
with $C_0^2(\mR^d) \subset D(\mathcal{A})$.

According to Theorem \ref{10_th_1}, we immediately have the following corollary.

\begin{corollary}
Under the conditions of Theorem \ref{10_th_1}.
For every $f \in C_0^2(\mR^d)$,  the function
                \begin{equation*}
                  u(t,y): = \mE \big{(} f(Y_{E^S_t})|Y_0 =y\big{)}, \quad t\geq 0,
                \end{equation*}
is the unique solution of the  equation
                \begin{equation*}
                 \left\{
                     \begin{array}{ll}
                       (k \partial_t + \partial_t^w) u(t,y) = (\mathcal{A} -a) u(t,y) +  a f(y), &  t>0, \\
                       u(0,y) = f(y),
                     \end{array}
                   \right.
                \end{equation*}
in the strong sense (i.e., in the sense of (i)-(iv) as in Theorem \ref{10_th_1}).
\end{corollary}

\begin{flushright}
$\Box$
\end{flushright}

{\bf Acknowledgements} {This work was supported in part by
the National Natural Science Foundation
 of China(Grant No.11401029) and by Teacher
 Research Capacity Promotion of Beijing Normal University Zhuhai.}


\ \newline

\end{document}